\documentclass[12pt]{amsart}
\usepackage[matrix,arrow]{xy}
\usepackage{amssymb}
\usepackage{amsmath}
\usepackage{amsfonts}
\usepackage{epsfig}
\usepackage{color}
\usepackage{epstopdf}
\usepackage{graphicx}
\usepackage{amsthm}
\usepackage{enumerate}
\usepackage[mathscr]{eucal}
\usepackage{verbatim}

\usepackage[bookmarks=true,hyperindex,pdftex,colorlinks,citecolor=blue,
linkcolor=blue,urlcolor=blue]{hyperref}
\usepackage{tikz}
\usetikzlibrary{matrix,arrows.meta}

\theoremstyle{plain}
\newtheorem{theorem}{Theorem}[section]
\newtheorem{cor}[theorem]{Corollary}
\newtheorem{prop}[theorem]{Proposition}
\newtheorem{lemma}[theorem]{Lemma}

\theoremstyle{definition}

\newtheorem{example}[theorem]{Example}

\newtheorem{question}[theorem]{Question}

\newtheorem{definition}[theorem]{Definition}

\newcommand{\R}{\mathbb{R}}
\newcommand{\N}{\mathbb{N}}

\newcommand{\Lin}{\mathcal{L}}

\newcommand{\F}{\mathcal{F}}
\newcommand{\eps}{\varepsilon}

\DeclareMathOperator{\spann}{span}

\DeclareMathOperator{\QNA}{QNA}

\DeclareMathOperator{\SA}{SNA}

\DeclareMathOperator{\A}{LipA}

\newcommand{\Lip}{{\mathrm{Lip}}_0}

\renewcommand{\subset}{\subseteq}

\title[Norm attaining Lipschitz maps toward vectors]
{Norm attaining Lipschitz maps toward vectors}

\author[G.~Choi]{Geunsu Choi}
\address[G.~Choi]{Department of Mathematics Education, Dongguk University, Seoul 04620, Republic of Korea}
\email{\texttt{chlrmstn90@gmail.com}}

\thanks{The author was supported by Basic Science Research Program through the National Research Foundation of
Korea(NRF) funded by the Ministry of Education, Science and Technology [NRF-2020R1A2C1A01010377].} 


\date{\today}                                           

\begin{document}

\begin{abstract}
We extend the recent result of G. Godefroy which concerns the existence of non-norm attaining Lipschitz maps in order to characterize the norm attainment toward vectors for Lipschitz maps in the general setting of underlying space. The main theorem of the present paper states that the existence of non-norm attaining Lipschitz maps toward vectors on a large class of metric spaces is characterized by the finite-dimensionality of range space. As an extension of his counterexample, some denseness results on norm attaining Lipschitz maps toward vectors are also shown.
\end{abstract}

\maketitle

\section{Introduction}

The present paper is basically motivated by the conventions of \cite{G2}. The norm attainment toward vectors is one of several notions of norm attaining Lipschitz maps to resolve the failure of denseness on arbitrary Banach space due to the result of \cite{KMS}. Without a doubt, the most natural way of defining the norm attainment for Lipschitz maps would be the case when there exists a specific pair of distinct points $(p,q)$ such that $\left\|f(p)-f(q)\right\|/d(p,q) = \|f\|$ for a given Lipschitz map $f$, provided the distance function $d$ and the Lipschitz norm $\|f\|$ of $f$. This notion, which we usually call by the \emph{strong norm attainment} for Lipschitz maps, naturally extends the original concept of norm attainment for bounded linear operators. However, as we have pointed out, even $\R$ fails for the set of strongly norm attaining Lipschitz maps to be dense in the space $\Lip(\R,\R)$ of all Lipschitz maps from $\R$ into itself vanishing at 0. As a matter of fact, there have been many approach to fix this drawback. For instance in \cite{G2, KMS}, they defined somewhat weaker notions of norm attaining Lipschitz maps, and in \cite{CCGMR} they generalized the domain space of strong norm attainment from a Banach space to a metric space in order to give positive answers on the denseness. A complete classification of those norm attainment notions can be found in \cite{CCM}. Among those, our main definition inspired by \cite{G2} is as follows. All metric spaces $M$ are assumed to be complete and pointed, and $Y$ always denotes a real Banach space.

\begin{definition}
A Lipschitz map $f \in \Lip(M,Y)$ is said to \emph{attain its norm toward} $y \in Y$ if there exists a sequence of distinct pairs $\{(p_n,q_n)\} \subset M \times M$ such that $\lim_n [f(p_n)-f(q_n)]/d(p_n,q_n) = y$ with $\|y\|=\|f\|$. Moreover, if such $y$ exists for $f$, we say $f \in \A(M,Y)$.
\end{definition}

The only difference above from the original statement in \cite{G2} is that we have extended the underlying space from a Banach space to a general metric space by following the leading motivation of \cite{CCGMR,G1} for strong norm attainment. It is evident that every Lipschitz map which attains its norm in the strong sense attains its norm toward a vector, and we have no interest when $Y$ is finite-dimensional owing to the compactness. The main result in \cite{G2} concerning the existence of non-norm attaining Lipschitz maps is stated below with a reformulation.

\begin{theorem}\cite{G2}\label{theorem:LipA-dense-fail}
There are Banach spaces $X$ and $Y$ Lipschitz isomorphic such that no Lipschitz isomorphism from $X$ onto $Y$ attains its norm toward any vectors.
\end{theorem}

A major significance of the preceding result is that even for the weakest concept of Lipschitz norm attainment, it may have a wide range of non-norm attaining ones. Later in \cite{CCM} the existence of non-norm attaining Lipschitz map toward vectors was generalized to an arbitrary case of domain and range spaces, and our primary aim is to extend these results in view of a generalized domain space. Meanwhile, it was also discovered that Theorem \ref{theorem:LipA-dense-fail} can be extended further to the denseness problem, which we refer to \cite[Example 3.6]{CCM}. Hence the denseness of norm attainment toward vectors will be our secondary interest on the topic.

\ \

Let us bring back to the detailed explanation of backgrounds. We use the standard notion of functional analysis throughout the article. In particular, for a Banach space $X$, $B_X$ and $S_X$ denote the unit ball and the unit sphere of $X$, respectively. $X^*$ stands for the topological dual space of $X$. We write by $\Lin(X,Y)$ the space of all bounded linear operators from $X$ into another Banach space $Y$. The operator norm of a bounded linear operator will be abbreviated by the Lipschitz norm $\|\cdot\|$ since its Lipschitz norm coincides with the operator norm.

As understanding the context requires the concept so-called Lipschitz-free space, we briefly introduce the notions and conventions here. We refer to \cite{G1} for who interested in this topic and one may find the contents of strong norm attainment for Lipschitz maps as well. Let $\delta: M \to \Lip(M,\R)^*$ be the point evaluation map, indeed distance-preserving, given by $\delta(x)(f) := f(x)$. We define the notion of \emph{Lipschitz-free space} over $M$ by
$$
\F(M) := \overline{\spann}\{\delta(x): x \in M\} \subseteq \Lip(M,\R)^*.
$$
It is well known that there is an isometric isomorphism between the space $\Lip(M,Y)$ and the space $\Lin(\F(M),Y)$, precisely by the element $f$ and $T_f$ with the relation $T_f(\delta(x)) = f(x)$, and the correspondence of $T_f$ will be used without any mentioning. We write the set of \emph{molecules} of $M$ by
$$
\operatorname{Mol}(M) := \left\{m_{p,q} := \frac{\delta(p)-\delta(q)}{d(p,q)} : (p,q) \in M \times M,\, p \neq q \right\} \subseteq S_{\F(M)},
$$
and it will be used freely without any proof that $\overline{\operatorname{co}}(\operatorname{Mol}(M)) = B_{\F(M)}$ where $\overline{\operatorname{co}}$ denotes the closed convex hull. 

Since a metric space is not a normed space in general, normally it is not natural to consider a summation of spaces. Nevertheless, we are still able to consider the ``direct sum'' of the family of metric spaces to deal with the stability result which was first considered in \cite{K}. Throughout the article, this concept will be used very frequently. Let $\{(M_i,d_i)\}_{i \in I}$ be a family of metric spaces. We write the \emph{metric sum} $\coprod_{i \in I} M_i$ to denote the disjoint union of family of metric spaces, by identifying the origins and endowing with the metric $d : M \times M \to \R$ given by
\begin{displaymath}
d(p,q):=\left\{\begin{array}{@{}cl}
\displaystyle \phantom{.} d_i(p,q) & \text{if } p,q \in M_i \\\\
\displaystyle \phantom{.} d_i(p,0) + d_j(0,q) & \text{if } p \in M_i, q \in M_j \text{ with } i \neq j.
\end{array} \right.
\end{displaymath}
It is shown in \cite{K} that there is an isometric isomorphism between the spaces
$$
\F \biggl( \coprod_{i \in I} M_i \biggr) = \biggl[ \bigoplus_{i \in I} \F(M_i) \biggr]_{\ell_1}.
$$

Finally, we shall provide another key notion introduced in \cite{CCJM}. If $X$ and $Y$ are both Banach spaces, then a bounded linear operator $T \in \Lin(X,Y)$ is said to \emph{quasi attain its norm} if there exist a sequence $\{x_n\} \subset B_X$ and a vector $y_0 \in Y$ such that $Tx_n$ converges to $y_0$ with $\|y_0\|=\|T\|$. We denote by $T \in \QNA(X,Y)$ if such sequence and vector exist. The concept of quasi norm attaining operator was a spin off of the norm attainment toward vectors for Lipschitz maps inspired by the intuition $\QNA(X,Y) = \A(X,Y) \cap \Lin(X,Y)$, and it was found in \cite{CCJM} that the quasi norm attainment completely characterizes the approximability of operators in terms of the Radon-Nikod\'ym property (RNP in short). In the next sections, we will find out how they are connected closely to each other.

\ \

The main purpose of this article is as follows. In Section \ref{section:char}, we give results on the characterization of norm attainment toward vectors for Lipschitz maps, in terms of other known conditions. We examine the equivalence associated to the quasi norm attaining operators as explained above. Mainly in the article, on the large class of metric spaces (including arbitrary infinite subsets of finite-dimensional Banach spaces) which we introduce later, we show that the finite-dimensionality of the range space is a sufficient and necessary condition for the set $\A$ to be identified by the whole space. This result allows us to produce arbitrarily many non-norm attaining Lipschitz maps toward vectors which maps to an infinite-dimensional Banach space. In Section \ref{section:dense}, we classify and list up the previously known results to obtain new consequences on the denseness of norm attaining Lipschitz maps toward vectors. This work extends Godefroy's first negative result as we know that it becomes also the first example which fails the denseness. Finally, in order to produce new examples including a non-Banach space one, we prove some stability results on the denseness for Lipschitz norm attainments.

\
\

\section{Characterization of norm attaining Lipschitz maps toward vectors}\label{section:char}

In this section, we look into the set of norm attaining Lipschitz maps toward vectors, in order to give an answer to how many there are non-norm attaining Lipschitz maps toward vectors in the setting of metric spaces, generalizing the idea of \cite{G2}. To begin with, we start with a very basic but fundamental result relating the notion of two sets $\A$ and $\QNA$. The proof is quite straightforward, but we give here for completeness.

\begin{lemma}\label{lemma:QNA}
If $f \in \A(M,Y)$, then $T_f \in \QNA(\F(M),Y)$.
\end{lemma}

\begin{proof}
Let $\{(p_n,q_n)\}$ be a sequence of distinct pairs in $M \times M$ and $y_0 \in Y$ be a vector with $\|y_0\|=\|f\|$ such that $[f(p_n)-f(q_n)]/d(p_n,q_n)$ converges to $y_0$. Then, it is routine to see that $T_f(m_{p_n,q_n})$ converges to $y_0$ and $m_{p_n,q_n} \in \operatorname{Mol}(M) \subseteq B_{\F(M)}$ for each $n \in \N$, which finishes the proof.
\end{proof}

As a consequence, some of the answer can be given right away as there are already results failing the equality about quasi norm attaining operators in \cite{CCJM}.

\begin{example}
We have the following examples.
\begin{enumerate}
\item[\textup{(a)}] Let $M$ be an infinite metric space. Then, $\A(M,c_0) \subsetneqq \Lip(M,c_0)$.
\item[\textup{(b)}] Let $M$ be either an infinite closed subset of $\R$ with measure 0 or the pointed natural number set $\N \cup \{0\}$. Then, $\A(M,Y) \subsetneqq \Lip(M,Y)$ for every infinite-dimensional Banach space $Y$.
\end{enumerate}
\end{example}

\begin{proof}
(a). It follows from $\F(M)$ is infinite-dimensional and that $\QNA(\F(M),c_0) \subsetneqq \Lin(\F(M),c_0)$ by \cite[Proposition 4.3]{CCJM}.

(b). It is a direct consequence of that $\F(M)$ is isometric to $\ell_1$ (see \cite{CCGMR}) and $\QNA(\ell_1,Y) \subsetneqq \Lin(\ell_1,Y)$ for every infinite-dimensional Banach space by \cite[Proposition 4.4]{CCJM}.
\end{proof}

A natural question that can be raised would be whether the converse of Lemma \ref{lemma:QNA} holds or not. We do not know the complete answer yet, but at least we can give a partial answer provided some additional condition on the range space.

\begin{question}\label{question:A-QNA}
Is there an operator $T_f \in \QNA(\F(M),Y)$ such that $f \notin \A(M,Y)$?
\end{question}

For a possible partial answer to Question \ref{question:A-QNA}, we first introduce the following convincing result. Recall that a point $x \in C$ is a \emph{denting point} in a bounded closed convex set $C$ if $x \notin \overline{\operatorname{co}}(C \setminus B(x,\eps))$ for all $\eps>0$ where $B(x,\eps)$ is the closed ball of radius $\eps$ centred at $x$. We shall denote by $\operatorname{dent}(C)$ the set of all denting points in $C$.

\begin{prop}
Let $M$ be a metric space and $Y$ be a Banach space. If $T_f \in \QNA(\F(M),Y)$ quasi attains its norm toward a vector $y \in \operatorname{dent}(B_Y)$, then $f \in \A(M,Y)$.
\end{prop}

\begin{proof}
Assume that $\|T_f\|=1$. Let $\{\mu_n\} \subseteq B_{\F(M)}$ be a sequence such that $T_f(\mu_n)$ converges to $y \in \operatorname{dent}(B_Y)$. Since $\overline{\text{co}}(\operatorname{Mol}(M)) = B_{\F(M)}$, we may write
$$
\mu_n = \sum_{j=1}^{m_n} a_{n,j} v_{n,j}
$$
where $v_{n,j} \in \operatorname{Mol}(M)$, $0 \leq a_{n,j} \leq 1$ and $\sum_{j=1}^{m_n} a_{n,j}=1$ for each $n \in \N$. We now claim that for each suitable $1\leq i_n \leq m_n$, the sequence $\{v_{n,i_n}\} \subseteq \operatorname{Mol}(M)$ is the desired element for the norm attainment. If not, we may assume that there exists $\eps>0$ such that for any choice of $\{i_n\}$ and $n \in \N$, $\|y-T_f(v_{n,i_n})\| > \eps$ by passing to a subsequence if necessary. Otherwise, by negation we have for any $\eps>0$, there exist suitable $n_0 \in \N$ and $1 \leq i_{n_0} \leq m_{n_0}$ such that $\|y-T_f(v_{n_0,i_{n_0}})\| \leq \eps$. This leads to that we may find appropriate $n_j$ and $1 \leq i_{n_j} \leq m_{n_j}$ for each $j \in \N$ such that $\|y-T_f(v_{n_j,i_{n_j}})\| \leq 1/2^j$, and hence $T_f(v_{n_j,i_{n_j}})$ converges to $y$. This shows that $f$ attains its norm toward $y$. 

Now by writing $v_{n,i_n}= m_{p_{n,i_n},q_{n,i_n}}$, we may find $\eps>0$ such that
\begin{equation}\label{equation:eps-distance}
\left\|y - \frac{f(p_{n,i_n})-f(q_{n,i_n})}{d(p_{n,i_n},q_{n,i_n})}\right\| > \eps
\end{equation}
for any $n \in \N$ and $1 \leq i_n \leq m_n$ by taking a suitable subsequence of $\{m_{p_{n,i_n},q_{n,i_n}}\}$ if necessary. Therefore, we obtain that the subsequence with an abuse of index
$$
T_f \left( \sum_{j=1}^{m_n} a_{n,j}v_{n,j} \right) = \sum_{j=1}^{m_n} a_{n,j} \frac{f(p_{n,j})-f(q_{n,j})}{d(p_{n,j},q_{n,j})}
$$
converges to $y$ while we have \eqref{equation:eps-distance} for all $n \in \N$ and $1 \leq j \leq m_n$, and this contradicts to the assumption that $y \in \operatorname{dent}(B_Y)$ since $y \in \overline{\operatorname{co}}(B_Y \setminus B(y,\eps))$.
\end{proof}

Note that it is not always true that $f \in \A(M,Y)$ attains its norm toward a denting point, as there are plenty of Banach spaces even without any extreme point in the unit ball. Still, there are many spaces which answers positively to Question \ref{question:A-QNA}. One may observe that the following corollary generalizes \cite[Proposition 3.7]{CCM} in further ways.

\begin{cor}\label{cor:str-cvx-KK}
Let $M$ be a metric space and $Y$ be a strictly convex Banach space with the Kadec-Klee property. Then, $f \in \A(M,Y)$ if and only if $T_f \in \QNA(\F(M),Y)$.
\end{cor}

\begin{proof}
According to \cite{LLT}, every point in $S_X$ is a denting point of $S_X$ if and only if $X$ is strictly convex and the weak topology and the norm topology coincide in $S_X$. The rest is clear.
\end{proof}

We now state the main result of this section which characterizes the dimension of range space in terms of the set of norm attaining Lipschitz maps toward vectors on metric spaces, a further generalization of \cite{CCM} with respect to the extended domain spaces. To do so, we define a crucial notion of a general metric space. Let us say a sequence with distinct elements $\{x_n\}$ in $M$ is \emph{pairly distanced} if there exists a sequence $\{r_n\} \subseteq [0,1]$ such that
\begin{enumerate}
\item[\textup{(i)}] $d(x_{2j-1},x_{2k-1}) \geq r_j d(x_{2j-1},x_{2j}) + r_kd(x_{2k-1},x_{2k}) \qquad \text{for all } j \neq k \in \N$,
\item[\textup{(ii)}] $d(x_{2j-1},x_{2k}) \geq r_j d(x_{2j-1},x_{2j}) + (1-r_k)d(x_{2k-1},x_{2k}) \qquad \text{for all } j \neq k \in \N$,
\item[\textup{(iii)}] $d(x_{2j},x_{2k}) \geq (1-r_j) d(x_{2j-1},x_{2j}) + (1-r_k)d(x_{2k-1},x_{2k}) \qquad \text{for all } j \neq k \in \N$.
\end{enumerate}
Geometrically, this means that each $r_jd(x_{2j-1},x_{2j})$-ball (resp. $(1-r_j)d(x_{2j-1},x_{2j})$-ball) of a sequence element $x_{2j-1}$ (resp. $x_{2j}$) is placed far from other balls. If $\{r_n\}$ can be chosen as a constant $r_n = r \in [0,1]$ for each $n \in \N$, then we say the sequence $\{x_n\}$ is \emph{$r$-pairly distanced}. In the next theorem, we will see that our characterization can be done for any infinite-dimensional range space, and it will be shown in Corollary \ref{cor:Euclidean} that subsets of finite-dimensional Banach spaces are main candidates of metric spaces on the domain. Obviously, usual vectors spaces are examples of metric spaces with a pairly distanced sequence. More concrete materials can be found in the following examples.

\begin{example}
Let $M$ be a metric space satisfying one of the following properties:
\begin{enumerate}
\item[\textup{(a)}] $M = \{0\} \cup \left\{ \dfrac{1}{3^n} : n \in \N \right\} \subseteq \R$.
\item[\textup{(b)}] \cite[Example 3.17]{CCGMR} $M = \displaystyle \bigcup_{n=2}^\infty \{0, p_n, q_n\} \subseteq c_0$ with $p_n=\left( 2 - \dfrac{1}{2^n}\right)e_n$ and $q_n = e_n + \left( 1 + \dfrac{1}{2^n} \right) e_1$.
\item[\textup{(c)}] \cite[Example 3.23]{CCGMR} $M$ is the metric sum of $M_n = \{0,p_n,q_n\}$ where $d(0,p_n)=d(0,q_n)=1+\dfrac{1}{2^n}$ and $d(p_n,q_n)=2$ for each $n \in \N$.
\end{enumerate}
Then, $M$ contains a pairly distanced sequence.
\end{example}

\begin{proof}
(a). By a direct computation, one can see that the sequence $\{x_n\}$ given by $x_n = \dfrac{1}{3^n}$ for each $n \in \N$ is $r$-pairly distanced for $3/4 \leq r \leq 1$. Observe that $M$ does not contain any $1/2$-pairly distanced sequence.

(b). The sequence $\{q_{n+1}\}$ is $1/2$-pairly distanced in $M$. Indeed, $d(q_j,q_k)=1$ for all $j \neq k \in \N \setminus \{1\}$ gives that all the requirements for $r_n=1/2$ can be easily verified.

(c). The sequence $\{p_1,q_1,p_2,q_2,\ldots\}$ is $1/2$-pairly distanced. Observe that $M$ does not contain any $r$-pairly distanced sequence unless $r = 1/2$.
\end{proof}

Now we are ready to prove the following characterization theorem on the existence of non-norm attaining Lipschitz maps toward vectors on a large class of metric spaces.

\begin{theorem}\label{theorem:Y-characterization}
Let $Y$ be a Banach space and $M$ be a metric space with a pairly distanced sequence. Then, $\A(M,Y) = \Lip(M,Y)$ if and only if $Y$ is finite-dimensional.
\end{theorem}

\begin{proof}
If $Y$ is finite-dimensional, by the compactness it is clear that $\A(M,Y) = \Lip(M,Y)$. So assume that $Y$ is an infinite-dimensional Banach space.

Suppose that $M$ has a pairly distanced sequence. By hypothesis, we may find sequences $\{x_n\}$ in $M$ and $\{r_n\} \subseteq [0,1]$ satisfying the assumption. Let $\{y_n\} \subset S_Y$ be a basic sequence of distinct vectors of $Y$. We define $\hat{f}: M \to Y$ by
\scriptsize$$
\hat{f}(x) := \sum_{n=1}^\infty \left( 1 -\frac{1}{2^n}\right) \big[ \max \{ r_n d(x_{2n-1},x_{2n}) - d(x_{2n-1},x), 0 \} - \max \{ (1-r_n) d(x_{2n-1},x_{2n}) - d(x_{2n},x), 0 \} \big] \, y_n,
$$\normalsize
which is well-defined. More precisely, given any $x \in M$, $\hat{f}(x)$ is of the form
\footnotesize$$
\hat{f}(x) = \left( 1 -\frac{1}{2^j}\right) \big[ \max \{ r_j d(x_{2j-1},x_{2j}) - d(x_{2j-1},x), 0 \} - \max \{ (1-r_j) d(x_{2j-1},x_{2j}) - d(x_{2j},x), 0 \} \big] \, y_j
$$\normalsize
for some $j \in \N$. Set $f(x) : = \hat{f}(x) - \hat{f}(0)$, and we claim that $f \in \Lip(M,Y)$ with $\|f\|=1$ and $f$ does not attains its norm toward any vector.

To see that $f \in \Lip(M,Y)$ with $\|f\| = 1$, let $p, q \in M$ be given. Observe that we may assume the case is reduced into at most 4 cases:
\begin{enumerate}
\item[\textup{(i)}] $p, q \in M_j$ for some $j \in \N$,
\item[\textup{(ii)}] $p \in M_j, q \in M_k$ for some $j \neq k \in \N$,
\item[\textup{(iii)}] $p \in M_j$ for some $j \in \N$, $q \notin M_n$ for any $n \in \N$,
\item[\textup{(iv)}] $p,q \notin M_n$ for any $n \in \N$,
\end{enumerate}
where $M_n := \{x \in M : d(x_{2n-1},x) < r_nd(x_{2n-1},x_{2n}) \text{ or } d(x_{2n},x) < (1-r_n)d(x_{2n-1},x_{2n})\}$. Assume here that $d(x_{2j-1},p) < r_j d(x_{2j-1},x_{2j})$ and $d(x_{2k},q) < (1-r_k) d(x_{2k-1},x_{2k})$ for simplicity, and the rest case can be done similarly.

$
$

(i). In this case, we have
\footnotesize\begin{align*}
\|f(p)-f(q)\| &= \left\| \left(1 -\frac{1}{2^j}\right) \bigl[ r_jd(x_{2j-1},x_{2j}) - d(x_{2j-1},q) + (1-r_j)d(x_{2j-1},x_{2j}) - d(x_{2j},q) \bigr] y_j \right\| \\
&= \left\| \left(1 - \frac{1}{2^j} \right) \bigl[ d(x_{2j-1},x_{2j}) - d(x_{2j-1},p) - d(x_{2j},q) \bigr] y_j \right\| \leq d(p,q).
\end{align*}\normalsize

(ii). By the condition of $p,q \in M$, we obtain
\begin{align*}
\|f(p)-f(q)\| &\leq \bigl[ r_jd(x_{2j-1},x_{2j})-d(x_{2j-1},p) \bigr] + \bigl[ (1-r_k)d(x_{2k-1},x_{2k})-d(x_{2k},q) \bigr] \\
&\leq d(x_{2j-1},x_{2k}) - d(x_{2j-1},p) - d(x_{2k},q) \\
&\leq d(p,q).
\end{align*}

(iii). It is clear that $\hat{f}(q)=0$. Since $q \notin M_j$, it follows that
\begin{align*}
\|f(p)-f(q)\| &= \left\| \left( 1-\frac{1}{2^j} \right) \bigl[ r_j d(x_{2j-1},x_{2j}) - d(x_{2j-1},p) \bigr] y_j \right\| \\
&\leq \left\| \left( 1-\frac{1}{2^j} \right) \bigl[ d(x_{2j-1},q) - d(x_{2j-1},p) \bigr] y_j \right\| \leq d(p,q).
\end{align*}

(iv). It is straightforward that $f(p)-f(q)=0$ in this case. In order to obtain the promised result, we consider $x_{2j-1}, x_{2j} \in M$ for each $j \in \N$ to derive
$$
\|f(x_{2j-1})-f(x_{2j})\| = \left( 1 -\frac{1}{2^j}\right) d(x_{2j-1},x_{2j}),
$$
and this leads to that $\|f\| =1$.

Now, suppose that $f$ attains its norm toward $y_0$ for some vector $y_0 \in S_Y$. As $\{y_n\}$ is a basic sequence, it follows that $y_0 = \sum_{n=1}^\infty a_n y_n$ for some coefficients such that the series converges. We may assume that $a_{n_0} \neq 0$ for some fixed $n_0$. From that $f$ attains its norm toward $y_0$, there is a sequence of distinct pairs $\{(p_n,q_n)\} \subseteq M \times M$ such that $\lim_n [f(p_n)-f(q_n)]/d(p_n,q_n) = y_0$, and hence at least one of $p_n,q_n$ must belong to $M_{n_0}$ passing to a subsequence. Without any loss of generality, $p_n \in M_{n_0}$ for all $n \in \N$ by admitting the abuse of index, and thus the remaining possibilities are either when (i)' $\lim_n d(p_n,q_n) = \infty$ or (ii)' $\lim_n d(p_n,q_n) < \infty$.

(i)'. This means that $q_n \notin M_{n_0}$ for sufficiently large $n$. Therefore, the coefficient $a_{n_0}$ is obtained by
$$
|a_{n_0}| \leq \lim_n \frac{1- \dfrac{1}{2^{n_0}}}{d(p_n,q_n)} =0,
$$
a contradiction.

(ii)'. We split the case into when (iii)' $q_n \in M_{n_0}$ for all $n \in \N$ or when (iv)' $q_n \notin M_{n_0}$ for all $n \in \N$ passing to a subsequence. Since it is clear that $f$ cannot attain its norm toward $y_0$ in (iii)', we shall prove (iv)'. For simplicity, assume $q_n \in M_{m_n}$ where $m_n \neq n_0$ for each $n \in \N$, and again that $d(x_{2n_0-1},p) < r_{n_0} d(x_{2n_0-1},x_{2n_0})$ and $d(x_{2m_n},q) < (1-r_{m_n}) d(x_{2m_n-1},x_{2m_n})$. An analogous argument of (ii) shows that
\begin{align*}
&\|f(p_n)-f(q_n)\| + \frac{1}{2^{n_0}} \bigl[ r_{n_0}d(x_{2n_0-1},x_{2n_0}) - d(x_{2n_0-1},p_n) \bigr] \\
&\leq \bigl[r_{n_0}d(x_{2n_0-1},x_{2n_0}) - d(x_{2n_0-1},p_n)\bigr] + \bigl[(1-r_{m_n})d(x_{2m_n-1},x_{2m_n}) - d(x_{2m_n},q_n)\bigr] \\
&\leq d(p_n,q_n).
\end{align*}
Therefore,
$$
\|f(p_n)-f(q_n)\| \leq d(p_n,q_n) - \frac{1}{2^{n_0}} \bigl[ r_{n_0}d(x_{2n_0-1},x_{2n_0}) - d(x_{2n_0-1},p_n) \bigr]
$$
for all $n \in \N$, and noting $\lim_n d(p_n,q_n)< \infty$ and $a_{n_0} \neq 0$ yields again that $f$ cannot attain its norm toward $y_0$, which contradicts to our assumption.
\end{proof}

An important consequence of Theorem \ref{theorem:Y-characterization} is that on a large class of metric spaces we can argue the existence of non-norm attaining Lipschitz maps toward vectors. For instance, if $M$ is a metric space which involves an infinite subset of the Euclidean metric structure, there always exists a non-norm attaining one whenever $Y$ is an infinite-dimensional Banach space.

\begin{cor}\label{cor:Euclidean}
Let $M$ be any infinite subset of a finite-dimensional Banach space and $Y$ be an infinite-dimensional Banach space. Then, $\A(M,Y) \subsetneqq \Lip(M,Y)$.
\end{cor}

\begin{proof}
We will prove the case when $M$ is a subset of the $m$-dimensional space $\ell_\infty^m$ endowed with the maximum norm. Let us show that $M$ contains a $1$-pairly distanced sequence. Suppose that $M$ does not have any limit point in $\ell_\infty^m$. Take any distinct points $p,q \in M$. Fix $x_1 :=p$, $x_2:=q$, and choose a proper closed ball $B_1=B(x_1,2d(x_1,x_2))$ in $\ell_\infty^m$ centred at $x_1$ with a radius $2d(x_1,x_2)$ so that $x_1,x_2 \in B_1$. Since $M$ is infinite and it does not contain any limit point, $M \setminus B_1$ must be infinite. It follows that one of the orthants, a generalization of quadrants for $m$-dimenison, centred at $x_1$ outside $B_1$ must have infinite points of $M$ in it. Let $x_1 + \mathcal{O}_1$ be the corresponding translated orthant centred at $x_1$, and let $\eps_m \in \{-1,1\}$ be the signs of each coordinate determined by the orthant $\mathcal{O}_1$. If $M \cap [x_1 + (2\eps_1d(x_1,x_2), \ldots, 2\eps_md(x_1,x_2)) + \mathcal{O}_1]$ has infinite points, then we may choose any distinct $x_3, x_4 \in M \cap [x_1 + (2\eps_1d(x_1,x_2), \ldots, 2\eps_md(x_1,x_2)) + \mathcal{O}_1]$ possibly with a reversed order such that the closed ball centred at $x_3$ with a radius $d(x_3,x_4)$ has no intersection with $B(x_1,d(x_1,x_2))$. Otherwise, one of the cyliners $\mathcal{C}_j := \{x_1 + (\eps_1z_1, \ldots, \eps_mz_m) \in \ell_\infty^m: 0 \leq z_k \leq 2d(x_1,x_2) \text { for } k \neq j,\, z_j \geq 4d(x_1,x_2)\}$ must contain infinitely many points in $M$, and we may select $x_3, x_4$ from there. Now we take another closed ball $B_2$ centred at $r_2$ with a radius $2k_2$, such that $B_1 \cup B(x_3,d(x_3,x_4)) \subseteq B(r_2, k_2)$. By an inductive procedure, we may find new $\mathcal{O}_2$ centred at $r_2$ in a natural way. Consequently, $\{x_n\}$ is our desired sequence in $M$.

Suppose now that $M$ has a limit point $x_0 \in \ell_\infty^m$. Similarly, there is an open (as the case with infinite points in the hyperplane is much easier) orthant $x_0 + \mathcal{O}_1$ such that $M \cap (x_0 + \mathcal{O}_1)$ has a limit point $x_0$ in $\ell_\infty^m$. Take any $p \in [M \cap (x_0 + \mathcal{O}_1)] \setminus \{x_0\}$, and we can choose a proper $q \in [M \cap (x_0 + \mathcal{O}_1)] \setminus \{x_0\}$ such that the closed neighborhood $B(p,d(p,q))$ does not contain $x_0$. Put $x_1:=p$ and $x_2:=q$. The next step is to find $x_3,x_4 \in [M \cap (x_0 + \mathcal{O}_1)] \setminus \{x_0\}$ very close to $x_0$ so that the ball $B(x_3,d(x_3,x_4))$ has no intersection with $B(x_1,d(x_1,x_2))$ and $\{x_0\}$. As $x_0$ is a limit point, this process can be done infinitely many times, and $\{x_n\}$ is the desired sequence in $M$. The proof of an arbitrary finite-dimensional space can be done in a similar way by considering instead the partitions of orthant splitted by the hyperplanes $\mathcal{H}_{i,j} := \{(z_1,\ldots,z_m) \in \R^m: \eps_i z_i = \eps_j z_j\}$ for each $i \neq j$, but we omit the detail of its proof as it is repetitive.
\end{proof}

Another consequence of Theorem \ref{theorem:Y-characterization} is that such characterization is always possible by supplementing at most countable additional points in any metric space.

\begin{cor}
Let $M$ be a metric space and $Y$ be an infinite-dimensional Banach space. Then, there exists a metric space $M'$ satisfying the following properties:
\begin{enumerate}
\item[\textup{(i)}] $\A(M',Y) \subsetneqq \Lip(M',Y)$.
\item[\textup{(ii)}] $M' \setminus M$ is countable.
\item[\textup{(iii)}] $M$ is a metric subspace of $M'$, that is, $M \subseteq M'$ with the induced metric.
\end{enumerate}
\end{cor}

\begin{proof}
Define $M'$ by the metric sum of two spaces $M$ and $M_0$, where $M_0$ consists of a pairly distanced sequence $\{x_n\}$ as in Theorem \ref{theorem:Y-characterization}. It is evident that $M' \setminus M$ is countable and $M \subseteq M'$ with the induce metric. Moreover, we have $\A(M',Y) \subsetneqq \Lip(M',Y)$ since $\{x_n\}$ is pairly distanced as well in $M'$.
\end{proof}

\
\

\section{Denseness of norm attaining Lipschitz maps toward vectors}\label{section:dense}

In this section, we provide the known examples of pairs of $(M,Y)$ such that $\A(M,Y)$ is dense in $\Lip(M,Y)$ to give an improvement of new pairs for denseness of positive and negative answers. Recall first that in \cite{CCGMR}, they organized a large pool of metric spaces $M$ such that norm attaining Lipschitz maps are dense in the strong sense regardless of the choice of $Y$, in particular when $\F(M)$ has a sufficient condition to the Lindenstrauss property A. Very recently in \cite{CGMR}, they constructed a metric space $M$ explicitly so that $\SA(M,Y)$ is dense in $\Lip(M,Y)$ for arbitrary Banach space $Y$. This answers some unsolved questions given in \cite{CCGMR}, as there was previously no specific example of metric spaces of denseness such that $\F(M)$ fails the RNP. We give here a partial list of examples of them as follows for convenience.

\begin{example}\cite{CCGMR,CGMR}
Let $M$ be a metric space with one of the following properties:
\begin{enumerate}
\item[\textup{(a)}] $M$ is uniformly discrete.
\item[\textup{(b)}] $M$ is countable compact.
\item[\textup{(c)}] $M$ is a compact H\"older metric space.
\item[\textup{(d)}] $M$ is a closed subset of $\R$ with measure zero.
\item[\textup{(e)}] $M$ is a metric space constructed as in \cite[Theorem 2.5]{CGMR}.
\end{enumerate}
Then, $\A(M,Y)$ is dense in $\Lip(M,Y)$ for every Banach space $Y$.
\end{example}

In the analogue of \cite{CCJM}, they improved the denseness result on the norm attainment for Lipschitz maps toward vectors in view of range spaces. Mainly they deal with the pairs of Banach spaces, and we state down the extended version of the result for the sake of completeness.

\begin{prop}
Let $M$ be a metric space and $Y$ be a Banach space with the RNP. Then, $\A(M,Y)$ is dense in $\Lip(M,Y)$.
\end{prop}

\begin{proof}
Since $Y$ has the RNP, by \cite[Corollary 3.10]{CCJM} the set of uniquely quasi norm attaining operators from $\F(M)$ into $Y$ is dense in $\Lin(\F(M),Y)$. Given any $T \in \Lin(\F(M),Y)$, we may find an operator $S$ approximating $T$ such that for every sequence $\{m_{p_n,q_n}\} \subseteq \operatorname{Mol}(M)$ with $\lim_n \|S(m_{p_n,q_n})\| = \|S\|$, there is a subsequence such that $S(m_{p_{\sigma(n)},q_{\sigma(n)}})$ converges to some vector $y_0$ with $\|y_0\|=\|S\|$. As we have $\overline{\operatorname{co}}(\operatorname{Mol}(M)) = B_{\F(M)}$, this completes the proof.
\end{proof}

On the other hand, as a direct consequence of Corollary \ref{cor:str-cvx-KK}, we have another denseness result for $\A(M,Y)$. However, we do not know whether the following result can provide a new result which has not been known.

\begin{cor}
Let $M$ be a metric space and $Y$ be a strictly convex Banach space with the Kadec-Klee property such that $\QNA(\F(M),Y)$ is dense in $\Lin(\F(M),Y)$. Then, $\A(M,Y)$ is dense in $\Lip(M,Y)$.
\end{cor}

We move on to the next step of the denseness, so-called the stability problem. There have been many efforts handling the heredity of norm attainment for operators and Lipschitz maps. We refer to \cite{CM, CDJM} for these kinds of study which have been done recently, and we basically follow their ideas. Recall that an \emph{absolute norm} $|\cdot|_a$ on $\R^2$ is a norm satisfying that (i) $|(1,0)|_a=|(0,1)|_a=1$, (ii) $|(p,q)|_a = |(|p|,|q|)|_a$ for all $p,q \in \R$. An \emph{absolute sum} $Y$ of Banach spaces $Y_1$ and $Y_2$ is a direct sum $Y_1 \oplus_a Y_2$ endowed with the absolute norm $\|\cdot\|_a$, where $\|\cdot\|_a$ is a norm satisfying that $\|(y_1,y_2)\|_a = |(\|y_1\|,\|y_2\|)|_a$ for all $y_1 \in Y_1$ and $y_2 \in Y_2$. A Banach space $Y_1$ is called an \emph{absolute summand} of $Y$ if $Y= Y_1 \oplus_a Y_2$ for some Banach space $Y_2$.

\begin{prop}\label{prop:absolute-sum}
Let $M$ be a metric space and $Y$ be a Banach space. If $Y_1$ is an absolute summand of $Y$ and $\A(M,Y)$ is dense in $\Lip(M,Y)$, then $\A(M,Y_1)$ is dense in $\Lip(M,Y_1)$.
\end{prop}

\begin{proof}
Let $Y = Y_1 \oplus_a Y_2$, $0<\eps<1$ and $f \in \Lip(M,Y_1)$ with $\|f\|=1$ be given. Define $\hat{f} \in \Lip(M,Y)$ by $\hat{f}(x) := (f(x),0)$. Then, there exists $\hat{g} \in \A(M,Y)$ with $\|\hat{g}\|=1$ such that $\|\hat{g}-\hat{f}\| <\eps$. Assume that $[\hat{g}(p_n)-\hat{g}(q_n)]/d(p_n,q_n)$ converges to $y_0$ for some sequence $\{(p_n,q_n)\}$ of distinct pairs in $M \times M$ and $y_0 \in S_Y$. Define $g_i:=\pi_i \circ \hat{g}$ and $y_i = \pi_i \circ y_0$ for $i=1,2$ where $\pi_i$ is a natural projection from $Y$ to $Y_i$ for $i=1,2$. Let $y^* = (y_1^*,y_2^*) \in S_{Y_1^* \oplus_{a^*} Y_2^*}$ be such that $y^* ((y_1, y_2)) = \|\hat{g}\|$. By \cite[Lemma 1.5]{CDJM}, we can derive that
\begin{equation}
y_1^*(y_1) + y_2^*(y_2) = \|y_1^*\| \|y_1\| + \|y_2^*\| \|y_2\| = 1.
\end{equation}
Note that $\|g_2\| < \eps$ since
\begin{align*}
\|g_2 \| & \leq \sup \left\{ \frac{\|[\hat{g}(p)-\hat{g}(q)] - (f(p)-f(q),0)\|}{d(p,q)} : (p,q) \in M \times M,\, p \neq q \right\} \\
&= \|\hat{g}-\hat{f}\| < \eps.
\end{align*}
This yields that
$$
\|y_1\| \geq \|y_0\| - \|y_2\| \geq \|y_0\| - \|g_2\| > 1-\eps >0.
$$
We now claim that
$$
g(x) := \|y_1^*\| \, g_1(x) + y_2^*(g_2(x)) \frac{y_1}{\|y_1\|}
$$
is the desired map attaining its norm toward $y_0$.

First, note that $[g(p_n)-g(q_n)]/d(p_n,q_n)$ converges to $y_1/\|y_1\|$ since
\begin{align*}
\frac{g(p_n)-g(q_n)}{d(p_n,q_n)} & = \|y_1^*\|\, \frac{g_1(p_n)-g_1(q_n)}{d(p_n,q_n)} + y_2^*\left(\frac{g_2(p_n)-g_2(q_n)}{d(p_n,q_n)}\right) \frac{y_1}{\|y_1\|}
\end{align*}
and that
$$
\frac{y_1}{\|y_1\|} \bigl[\|y_1^*\| \|y_1\| + y_2^*(y_2)\bigr] = \frac{y_1}{\|y_1\|},
$$
hence $\|g\| \geq 1$. To deduce $\|g\| = 1$, observe that
\begin{align*}
\|g\| & \leq \sup \left\{ \|y_1^*\| \frac{\|g_1(p)-g_1(q)\|}{d(p,q)} + \|y_2^*\| \frac{\|g_2(p)-g_2(q)\|}{d(p,q)} : (p,q) \in M \times M,\, p \neq q \right\} \\
& \leq \sup \left\{ \left| \left( \frac{\|g_1(p)-g_1(q)\|}{d(p,q)}, \frac{\|g_2(p)-g_2(q)\|}{d(p,q)} \right) \right|_a : (p,q) \in M \times M,\, p \neq q \right\} \\
&= \| \hat{g} \|.
\end{align*}
It remains to prove that $g$ is close enough to $f$. From the fact that $\|y_1^*\| \|y_1\| + \|y_2^*\| \|y_2\| = 1$, it is easy to see that $\|y_1^*\| > 1 - \eps$. Thus we have
$$
\|g-f\| \leq \|y_2^*\| \| g_2 \| + \bigl\|\|y_1^*\|g_1 - g_1 \bigr\| + \| g_1 -f \| < 3\eps,
$$
which finishes the proof.
\end{proof}

Wih the aid of the concept of metric sum, we are able to prove the stability equivalence on denseness with respect to the domain space in this time.

\begin{theorem}\label{theorem:ell_1-sum}
Let $\{M_i\}_{i \in I}$ be a family of metric spaces, $M = \coprod_{i \in I} M_i$ and $Y$ be a Banach space. Then, the following are equivalent.
\begin{enumerate}
\item[\textup{(a)}] $\A(M_i,Y)$ is dense in $\Lip(M_i,Y)$ for every $i \in I$.
\item[\textup{(b)}] $\A(M,Y)$ is dense in $\Lip(M,Y)$.
\end{enumerate}
\end{theorem}

\begin{proof}
(a)$\Rightarrow$(b). Let $\eps>0$ and $f \in \Lip(M,Y)$ be given. if we write $E_i : \F(M_i) \to \F(M)$ to be a natural injection, we have that
$$
\|T_f\| = \sup_{i \in I} \|T_f E_i\|.
$$
Choose $h \in I$ so that $\|T_f E_h\| > \|f\|-\eps$. By assumption, there exists $g_0 \in \A(M_h,Y)$ such that $\|g_0\| = \|T_f E_h\|$ by rescaling and $\|T_{g_0} - T_f E_h\|<\eps$. Assume that $[g_0(p_n)-g_0(q_n)]/d(p_n,q_n)$ converges to $y_0$ for some sequence $\{(p_n,q_n)\}$ of distinct pairs in $M_h \times M_h$ and $y_0 \in Y$ with $\|y_0\|=\|g_0\|$. Let us define $T_g \in \Lin(\F(M),Y)$ by
\begin{displaymath}
T_g E_i=\left\{\begin{array}{@{}cl}
\displaystyle \phantom{.} T_{g_0} & \text{if } i=h \\
\displaystyle \phantom{.} (1-\eps) T_f E_i & \text{if } i \neq h.
\end{array} \right.
\end{displaymath}
Then, $\|g\|=\|T_g\|=\|T_{g_0}\|=\|g_0\|$ and
$$
\|g-f\| = \sup_{i \in I} \|(T_g-T_f)E_i\| \leq \eps.
$$
Finally, if we consider the sequence $\{E_h m_{p_n,q_n}\} \subset \operatorname{Mol}(M)$, then $T_g (E_h m_{p_n,q_n})$ converges to  $y_0$ and $\|y_0\|=\|g\|$, which concludes that $g \in \A(M,Y)$.

\ \

(b)$\Rightarrow$(a). Let $0<\eps<1$ and $f \in \Lip(M_h,Y)$ be given with $\| f \| = 1$ for some $h \in I$. If we consider $T_f P_h \in \Lin(\F(M),Y)$ where $P_h: \F(M) \to \F(M_h)$ is the natural projection, then we have that $\|T_f P_h\| = \|f\|$. By assumption, there exists $g_0 \in \A(M,Y)$ such that $\|g_0\|=\|T_f P_h\|$ and $\|T_{g_0} - T_f P_h\| < \eps$. Assume that $[g_0(p_n)-g_0(q_n)]/d(p_n,q_n)$ converges to $y_0$ and $\|y_0\|=\|g_0\|$ for some sequence $\{(p_n,q_n)\}$ of distinct pairs in $M \times M$ and $y_0 \in Y$. Note that if $m_{p_n, q_n} \in \F (M) \setminus \F (M_h)$ for some $n \in \mathbb{N}$, then 
$$
\left\|\frac{g_0(p_{n})-g_0(q_{n})}{d(p_{n},q_{n})}\right\| \leq \|T_{g_0} - T_f P_h\| < \eps. 
$$
This implies that we can find a subsequence $\{(p_{n_j},q_{n_j})\} \subset M \times M$ satisfying that $m_{p_{n_j}, q_{n_j}} \in \F (M_h)$ and $T_{g_0} (E_h m_{p_{n_j},q_{n_j}})$ converges to $y_0$. Let $g \in \Lip (M_h, Y)$ be such that $T_g =T_{g_0} E_h \in \Lin(\F (M_{h}),Y)$. Then,
$$
\|g-f\|=\|T_{g_0}E_h - T_f P_h E_h\|<\eps
$$
and $T_g (m_{p_{n_j},q_{n_j}}) = T_{g_0} (E_h m_{p_{n_j},q_{n_j}})$ converges to $y_0$.
\end{proof}

One may find new pairs of spaces which fail the denseness by applying Proposition \ref{prop:absolute-sum} or Theorem \ref{theorem:ell_1-sum}. Particularly, Theorem \ref{theorem:ell_1-sum} allows us to produce a non-Banach space example of metric space such that the denseness of norm attaining Lipschitz maps toward vectors fails.

\begin{cor}
There exist a metric space $M$ which is not a Banach space and a Banach space $Y$ such that $\A(M,Y)$ is not dense in $\Lip(M,Y)$.
\end{cor}

\begin{proof}
By Theorem \ref{theorem:ell_1-sum}, for Banach spaces $X$ and $Y$ such that $\A(X,Y)$ is not dense in $\Lip(X,Y)$, the space $M= \coprod_{i=1}^2 M_i$ with $M_1 = X$ satisfies that $\A(M,Y)$ is not dense in $\Lip(M,Y)$. Choosing $M_2$ to be a metric space which is not a Banach space gives the desired space.
\end{proof}

\ \


\end{document}